\documentclass{article}

\usepackage{arxiv}

\usepackage[utf8]{inputenc} 
\usepackage[T1]{fontenc}    
\usepackage{hyperref}       
\usepackage{url}            
\usepackage{booktabs}       
\usepackage{amsfonts}       
\usepackage{nicefrac}       
\usepackage{microtype}      
\usepackage{lipsum}
\usepackage{graphicx}
\graphicspath{ {./images/} }
\usepackage{amsthm, amsfonts, amsmath, amssymb}

\usepackage{enumerate}

\usepackage[all,2cell]{xy} \UseAllTwocells \SilentMatrices

\newtheorem{theorem}{Theorem}[section]
\newtheorem{proposition}[theorem]{Proposition}
\newtheorem{corollary}[theorem]{Corollary}
\newtheorem{lemma}[theorem]{Lemma}

\newtheorem{definition}[theorem]{Definition}
\newtheorem{example}[theorem]{Example}
\newtheorem{remark}[theorem]{Remark}

\usepackage[overload]{empheq}
\usepackage{nccmath}

\usepackage{quiver}
\usepackage{tikz-cd}
\usepackage{ dsfont }
\usepackage{mathtools}

\usepackage{tikz}
\usetikzlibrary{matrix}
\usetikzlibrary{calc}
\usepackage{blindtext}
\usepackage{authblk} 

\title{The Category of Hypergroups: seeking for a \\ generalized version of Abelian Categories}

\author[1]{Ana Luiza Tenório}
\author[2]{Kaique Matias de Andrade Roberto}

\affil[1]{Department of Mathematics, University of São Paulo, \texttt{analuiza@ime.usp.br}}
\affil[2]{Department of Mathematics, University of São Paulo, \texttt{kaique.roberto@usp.br}}

\begin{document}
\maketitle
\begin{abstract}{
 In this paper we study categorical properties of the category of (abelian) hypergroups that leads to the notion of hyper-quasipreadditive and hyper-quasiabelian categories. Our goal is to create a path towards a general theory of homological algebra for hyperalgebras.  This is a first attempt to achieve this goal. We hope to improve the definitions and results, and provide more examples soon. }

\end{abstract}


\section{Introduction}

The concept of multialgebraic structure -- an ``algebraic like'' structure but endowed with  multiple valued operations -- has 
been studied since the 1930s; in particular, the concept of hypergroup was introduced by Marty in 1934 and the concept of hyperrings was introduced by Krasner in the 1950s. 

Some general algebraic studies have been conducted on multialgebras (\cite{golzio2018brief} and \cite{pelea2006multialgebras}). In particular, for multirings/hypergroups, various properties and applications have been explored. We would like to cite those that are related to our research group, including applications in abstract quadratic forms theory (\cite{marshall2006real}, \cite{worytkiewiczwitt2020witt}, \cite{roberto2021quadratic}, \cite{roberto2021ktheory}, \cite{ribeiro2021anel}), tropical geometry (\cite{jun2015algebraic}), and a more detailed discussion on variants of the concept of polynomials over hyperrings (\cite{jun2015algebraic}, \cite{ameri2019superring}). For a comprehensive overview of the foundations of hypergroup theory and numerous references to applications, we recommend the article \cite{massouros2021overview}.

In this present work, we address the basic algebraic and categorical aspects of hypergroups, with the following objectives: (1) introducing the notion of a hyper preadditive category; (2) developing a theory of (hyper) homology/cohomology over hyperstructures; and (3) establishing a theory of sheaves over hyperstructures. Our aim is to explore the entire structure of hypergroups in the most general setting possible, without attempting to obtain a group from a hypergroup.

It is important to note that we do not attempt to provide an in-depth analysis of all three aspects mentioned above. Instead, our goal is to present a broad framework in which (almost) all hyperalgebras can be inserted and studied. We also acknowledge that a similar idea was previously explored in \cite{ameri2003categories}, although our definition of hypergroups differs slightly.

The structure of the paper is as follows: We begin by presenting the definition of hypergroups we are considering, along with key examples and basic properties. In Section \ref{sec:categorical_properties}, we study the category of hypergroups, focusing on its (co)kernels, (co)products, and directed (co)limits. Subsequently, in Section \ref{sec:hyperbolic}, we narrow our focus to hyperbolic hypergroups, as their corresponding category exhibits more desirable properties (for instance, we prove that morphisms of hyperbolic hypergroups are monomorphisms precisely when they are injective).

Moving on to Section \ref{sec:hyper-abelian}, we introduce the concepts of hyper-almost-preadditive and hyper-almost-abelian categories. The category of abelian hypergroups serves as an example of such abstract categories. We also briefly discuss the behavior of injective objects and exact sequences within this context. Finally, the last section addresses the interaction between the theory we are constructing, recent developments utilizing sheaves of hyperrings, and potential avenues for future research.

\section{Hypergroups}

The content of this section is quite standard in the literature of hypergroups (and hyperrings/hypergroups). However, at this moment there are no unified notation for these objects/concepts. We refer to \cite{massouros2021overview} in order to get an overview of the theory. 

Hypergroups are a generalization of groups. We can think that a hypergroup is a group with a multivalued operation:
\begin{definition}\label{definition:multigroupI}
 A hypergroup is a quadruple $(G,\ast,r,1)$, where $G$ is a non-empty set, $\ast:G\times G\rightarrow\mathcal 
P(G)\setminus\{\emptyset\}$\footnote{Here $\mathcal P(G)$ is the \textbf{power set} of $G$, i.e, $\mathcal P(G):=\{A:A\subseteq G\}$.} (we denote $x\ast y:=\ast(x,y)$) and $r:G\rightarrow G$
 are functions, and $1$ is an element of $G$ satisfying:
 \begin{enumerate}[i -]
  \item If $ z\in x\ast y$ then $x\in z\ast r(y)$ and $y\in r(x)\ast z$.
  \item $y\in 1\ast x$ if and only if $x=y$.
  \item With the convention $x\ast(y\ast z)=\bigcup\limits_{w\in y\ast z}x\ast w$ and 
  $(x\ast y)\ast z=\bigcup\limits_{t\in x\ast y}t\ast z$,
  $$x\ast(y\ast z)=(x\ast y)\ast z\mbox{ for all }x,y,z\in G.$$
  
	A hypergroup is said to be \textbf{commutative} if
  \item $x\ast y=y\ast x$ for all $x,y\in G$.
 \end{enumerate}
\end{definition}

Observe that by (i) and (ii), $1\ast x=x\ast 1=\{x\}$ for all $x\in G$. When $a\ast b=\{x\}$ be a unitary set, we just write 
$a\ast b=x$.

For $a\in G$, we also denote $r(a):=a^{-1}$. Combining (i) and (ii), we get for all $a\in G$ that $a\in1\ast a$, $1\in a\ast a^{-1}$ and if $1\in a\ast b$ then $b=a^{-1}$. 

From now on, we just denote $a\ast b=a\cdot b=ab$.

\begin{example}
Every group $(G,\cdot,1)$ is a hypergroup if we define the multi-operation by $a\ast b=\{a\cdot b\}$.
\end{example}

\begin{example}[\cite{gratzer1962representation}]
Let $G$ be a group and $H\subseteq G$ be a subgroup (not necessarily normal). Define
$$G/H:=\{aH:a\in G\}.$$
In other words, $G/H$ is the set of cosets which, in this case, does not arise from an equivalence relation in general. We denote the elements in $G/H$ simply by $[a]:=aH$. Now, for $[a],[b]\in G/H$, define
\begin{align*}
  [a]\ast[b]&=\{[d]\in G/H:\mbox{ there exist }a',b',d'\in G \\
  &\mbox{ with }[d']=[d],\,[a']=[a],\,[b']=[b]\mbox{ and }d'\in a'\ast b'\}.  
\end{align*}
Then $(G/H,\ast,[1])$ is a hypergroup which is abelian if $G$ is abelian.
\end{example}

\begin{example}[Multigroup of a Linear Order, 3.4 of \cite{viro2010hyperfields}]
Let $(\Gamma,\cdot,1,\le)$ be an ordered abelian group. We have an associated hypergroup structure $(\Gamma\cup\{0\},+,-\cdot,0,1)$ with the rules $-a:=a$, $a\cdot0=0\cdot a:=0$, the convention $0\le a$ for all $a\in\Gamma$,  and
$$a+b:=\begin{cases}a &\mbox{ if }a<b \\ b &\mbox{ if }b<a\\
[0,a] &\mbox{ if }a=b.\end{cases}$$
\end{example} 

\begin{example}[The Hypergroup $Q_2$]\label{q2}
$Q_2=\{-1,0,1\}$ is a hypergroup with the multivalued sum defined by 
relations
  $$\begin{cases}
     0+x=x+0=x,\,\mbox{for every }x\in Q_2 \\
     1+1=1,\,(-1)+(-1)=-1 \\
     1+(-1)=(-1)+1=\{-1,0,1\}.
    \end{cases}
  $$
  If we consider in $Q_2$ the usual product (in $\mathbb Z$), then we have a hyperfield.
\end{example}

 \begin{example}\label{exk}
 Let $K=\{0,1\}$ with the sum defined by relations $x+0=0+x=x$, $x\in K$ and
$1+1=\{0,1\}$. This is a hypergroup, called Krasner's hypergroup (in fact, it is also a hyperfield, and for more details the reader can consult \cite{jun2015algebraic}).
\end{example}

For hypergroups there are other examples of interest. For more details see, for instance, \cite{massouros2021overview}. For hyperrings/hypergroups and applications see for instance \cite{ribeiro2016functorial}, \cite{worytkiewiczwitt2020witt} or \cite{viro2010hyperfields}. Here are the basic properties holding in every hypergroup (for more details and proofs, see for example \cite{ribeiro2016functorial}, \cite{viro2010hyperfields}, \cite{ribeiro2021anel}).
\begin{lemma}\label{hlem1}
Let $G$ be a hypergroup and $a,b,c,d\in G$. Then:
\begin{enumerate}[i -]
    \item $(1)^{-1}=1$;
    \item $(a^{-1})^{-1}=a$;
    \item $c\in ab$ if and only if $c^{-1}\in a^{-1}b^{-1}$;
    \item $(ab)^{-1}=b^{-1}a^{-1}$.
\end{enumerate}
\end{lemma}

$ $




\begin{definition}
 Let $G$ and $H$ be hypergroups. A map $f:G\rightarrow H$ is a \textbf{morphism} if for all $a,b,c\in G$:
 \begin{enumerate}[i -]
  \item $c\in a\ast b\Rightarrow f(c)\in f(a)\ast f(b)$;
  \item $f(a^{-1})=(f(a))^{-1}$;
  \item $f(1)=1$.
 \end{enumerate}
 The morphism $f$ is \textbf{full} if for all $a,b\in G$,
 $f(a\ast b)=f(a)\ast f(b)$.
\end{definition}


There is another description of hypergroups due to M. Marshall\footnote{This is a first-order theory with
axioms of the form $\forall\exists$.}:

\begin{definition}[\cite{marshall2006real}]\label{definition:multigroupII}
 A hypergroup is a quadruple $(G,\Pi,r,\mathfrak i)$ where $G$ is a non-empty set, $\Pi$ is a subset of 
$G\times 
G\times G$, $r:G\rightarrow G$ is a
 function and $\mathfrak{i}$ is an element of $G$ satisfying:
 \begin{enumerate}[i -]
  \item If $(x,y,z)\in\Pi$ then $(z,r(y),x)\in\Pi$ and $(r(x),z,y)\in\Pi$.
  \item $(x,\mathfrak{i},y)\in\Pi$ if and only if $x=y$.
  \item If $\exists\, p\in G$ such that $(u,v,p)\in\Pi$ and $(p,w,x)\in\Pi$ then $\exists\, q\in G$ such 
that 
$(v,w,q)\in\Pi$ and $(u,q,x)\in\Pi$.
	
	A hypergroup is said to be \textbf{commutative or abelian} if
  \item $(x,y,z)\in\Pi$ if and only if $(y,x,z)\in\Pi$.
 \end{enumerate}
\end{definition}

For multi-structures/hyper-structures, there are various sorts of ``substructure'' that one can consider. For simplicity, we will deal only with subhypergroups and full subhypergroups,  defined below. For more details, we suggest \cite{ribeiro2016functorial}, \cite{ribeiro2021anel}, \cite{golzio2018brief} or even \cite{massouros2021overview}.

\begin{definition}
 If $G,H$ are hypergroups with $H\subseteq G$, we say $H$ is \textbf{a subhypergroup (or just subgroup)} of $G$ if the inclusion $H\hookrightarrow G$ is a morphism. We say $H$ is a \textbf{full subhypergroup (or just full subgroup)} of $G$ if the inclusion $H\hookrightarrow G$ is a full morphism. Note that in the group case, all these prescriptions coincide. 
\end{definition}

Let $\varphi:G\rightarrow H$ be a morphism of hypergroups. Denote
$$\mbox{Ker}(\varphi):=\{g\in G:\varphi(g)=1\}\mbox{ and }
\mbox{Im}(\varphi)=\{h\in H:h=\varphi(g)\mbox{ for some }g\in G\}.$$
Then $\mbox{Ker}(\varphi)$ is a full subgroup of $G$ and $\mbox{Im}(\varphi)$ is a subgroup of $H$, but $\mbox{Im}(\varphi)$ (with the multi operation inherited from $H$) is not a full subgroup in general.

\begin{definition}[Generated Subgroup]
 Let $G$ be a hypergroup and $A\subseteq G$. We define the \textbf{subgroup generated by }$A$, notation $\langle A\rangle$ by
 $$\langle A\rangle:=\bigcap\{H:H\subseteq G\mbox{ is a subgroup and }A\subseteq H\}.$$
\end{definition}

Let $G$ be a hypergroup and $a\in G$. For $n\ge0$, denote $a^n:=a\cdot...\cdot a$ ($n$ times) and $a^{-n}:=(a^{-1})^n$.

\begin{proposition}\label{prop1}
 Let $G$ be a hypergroup and $a\in G$. Then
 $$\langle a\rangle=\bigcup\{a^{i_1}\cdot...\cdot a^{i_n}:i_j\in\mathbb Z,\,j=1,...,n,\,n\in\mathbb N\}.$$
\end{proposition}
\begin{proof}
Denote 
$$((a)):=\bigcup\{a^{i_1}\cdot...\cdot a^{i_n}:i_j\in\mathbb Z,\,j=1,...,n,\,n\in\mathbb N\}.$$
In fact, $1\in a\cdot a^{-1}\in((a))$ and if $x\in a^{i_1}\cdot...\cdot a^{i_n}$ then $x^{-1}\in a^{-i_n}\cdot...\cdot a^{-i_1}\in((a))$ (use item (iii) of Lemma \ref{hlem1} and induction). Finally, let $x\in a^{i_1}\cdot...\cdot a^{i_n}$ and $y\in a^{i_{n+1}}\cdot...\cdot a^{i_m}$. Then
$$xy\in (a^{i_1}\cdot...\cdot a^{i_n})(a^{i_{n+1}}\cdot...\cdot a^{i_m})=a^{i_1}\cdot...\cdot a^{i_n}\cdot a^{i_{n+1}}\cdot...\cdot a^{i_m}\subseteq((a)).$$
Then $((a))$ is a subgroup and since $a\in((a))$, we have $\langle a\rangle\subseteq((a))$. Moreover, if $H\subseteq G$ is a subgroup and $a\in H$, then $a^{-1}\in H$. Hence $((a))\subseteq H$ and $((a))\subseteq\langle a\rangle$.
\end{proof}

For abelian hypergroups, there is an easier description of $\langle a\rangle$:
\begin{proposition}
 Let $G$ be an abelian hypergroup and $a\in G$. Then
 $$\langle a\rangle:=\bigcup_{i,j\in\mathbb Z}\{a^{i}\cdot a^{-j}\}.$$
\end{proposition}
\begin{proof}
Let $a^{i_1}\cdot...\cdot a^{i_n}$ with $i_1,...,i_n\in\mathbb Z$. Write
$\{i_1,...,i_n\}:=\{s_1,...,s_k,-t_{k+1},...,-t_{n}\}$ with $s_1,...,s_k$ and $t_{k+1},...,t_n$ non negative integers. Then
$$a^{i_1}\cdot...\cdot a^{i_n}=a^{s_1+...+s_k}\cdot a^{-(t_{k+1}+...+t_n)}.$$
\end{proof}

Of course, in the abelian case, for $i\ge j$ we have $a^{i-j}\subseteq a^{i}\cdot a^{-j}$ but the equality does not hold in general. If $A\subseteq G$, adapting the argument used in Proposition \ref{prop1} we get
$$\langle A\rangle=
\bigcup\{a_1^{i_1}\cdot...\cdot a_n^{i_n}:a_j\in A,\,i_j\in\mathbb Z,\,j=1,...,n,\,n\in\mathbb N\}.$$

Let $(G,+,0)$ be an abelian hypergroup and $H\subseteq G$ be a full subgroup. Denote 
$$G/H:=\{a+H:a\in G\}$$
and $[a]:=a+H$, $a\in G$. Define $-[a]:=[-a]$ and
$$[a]+[b]:=\{[d]:d'\in a'+b'\mbox{ for some }[d']=[d],\,[a']=[a],\,[b']=[b]\}.$$
In fact, such prescriptions does not depend of the choice of representatives.

\begin{theorem}
Given the notations above, $(G/H,+.-,[0])$ is an abelian hypergroup. Moreover we have a morphism $\pi:G\rightarrow G/H$.
\end{theorem}
\begin{proof}
In fact, if $[a]\in[0]+[b]$, then $a'\in x+b'$ with $x\in H$, $[a']=[a]$, and $[b']=[b]$. Then $b'\in a'-x$,
$$a'+H\subseteq b'+x+H=b'+H\mbox{ and }b'+H\subseteq a'-x+H=a'+H.$$
Hence $[a]=[a']=[b']=[b]$. Since $0\in a-a$ we have $[0]\in[a]+[-a]=[a]-[a]$. Now, let $[x]\in([a]+[b])+[c]$. Then $[x]\in[d]+[c]$ for some $[d]\in[a]+[b]$. Moreover $x'\in d'+c'$ for some $d'\in a'+b'$ with $[x']=[x]$, $[a']=[a]$, $[b']=[b]$, $[c']=[c]$ and $[d']=[d]$. In this case, $x'\in(a'+b')+c'=a'+(b'+c')$, and then, $x'\in a'+e$ with $e'\in b'+c'$. Hence $[x]\in[a]+[e]$ with $[e]\in[b]+[c]$, and $[x]\in[a]+([b]+[c])$, which means $([a]+[b])+[c]\subseteq([a]+[b])+[c]$. With an analogous argument we conclude $([a]+[b])+[c]\subseteq([a]+[b])+[c]$.
\end{proof}

If $H\subseteq G$ is not a full subgroup, we define $G/H:=G/\langle H\rangle$, where $\langle H\rangle$ is the subgroup of $G$ generated by $H$. Note that we always have an equivalence relation defined for $a,b\in G$ by
$$a\equiv_Hb\mbox{ iff }a-b\subseteq H.$$
Note that if $a\equiv_Hb$ then $a+H=b+H$ but not necessarily the converse (even if $H$ is full).

\section{Categorical Properties of Hypergroups}\label{sec:categorical_properties}

Now we begin the investigation of categorical properties in the category of hypergroups. We start with the entire category of hypergroups (i.e, without supposing full morphisms) and gradually introduce subcategories (abelian hypergroups, abelian hypergroups with full morphisms, etc) in order to obtain stronger results.

\begin{proposition}
 The category of hypergroups has a terminal object.
\end{proposition}
\begin{proof}
Consider $\textbf{1} = \{1\}$ the trivial hypergroup. For any hypergroup $G$ we define a function $!: G \to \textbf{1}$ by $!(g) = 1$ for all $g \in G$. This is a morphism of hypergroups since:
\begin{enumerate}
    \item Take $c \in a * b$, for $a,b \in G$. We have $!(c) \in !(a) *!(b)$ because $1 \in 1 * 1$.
    \item $!(r(a)) = 1,$ for all $a \in G$ and $r(!(a)) = r(1) = 1$.
    \item $!(1) = 1$.
\end{enumerate}
And $!$ is the unique morphism of hypergroup with codomain $\textbf{1}$. If $\varphi: G \to \textbf{1}$ is another morphism of hypergroups, then $\varphi(g) = 1 = !(g)$ since $1$ is the unique element of \textbf{1}.
\end{proof}

\begin{proposition}
 The category of hypergroups has an initial object.
\end{proposition}
\begin{proof}
We will show that the trivial hypergroup $\textbf{1}$ is an initial object in the category of hypergroups. For any hypergroup $G$, define $\phi : \textbf{1} \to G$ by $\phi(1) = 1$. So the third condition is automatically satisfied. Since the unique element in $\{1\}$ is an identity element $1$, the other two conditions are easy to check:
\begin{enumerate}
    \item $1 \in 1 * 1 \mbox{ implies } \phi(1) \in \phi(1) * \phi(1)$.  
    \item $\phi(r(1)) = \phi(1) = 1 = r(1) = r(\phi(1))$.
\end{enumerate}
Observe that any other morphism $\psi: \textbf{1} \to G$ of hypergroups must satisfy the third condition $\psi(1) = 1$  so $\psi = \phi$. In other words, $\phi$ is unique.
\end{proof}

Since $\textbf{1}$ is initial and terminal, we conclude:
\begin{corollary}
The category of hypergroups has a zero object.
\end{corollary}

The product in the category of hypergroups is obtained in a very similar way to the product in the category of groups: if $\{G_i\}_{i\in I}$ is a family of hypergroups, for $(a_i)_{i\in I},(b_i)_{i\in I}\in\prod_{i\in I}G_i$ we define
$$(a_i)_{i\in I}\ast(b_i)_{i\in I}=\left\lbrace(c_i)_{i\in I}\in\prod_{i\in I}G_i:c_i\in a_i\ast b_i\mbox{ for all }i\in I\right\rbrace.$$

In fact, $(\prod_{i\in I}G_i,\ast,(1)_{i\in I})$ is a hypergroup. If we take, for each $i\in I$, the morphism $\pi_i:\prod_{i\in I}G_i\rightarrow G_i$ given by the rule $\pi_i(a_i)_{i\in I}:=a_i$, we have that $\pi_i$ is a full surjective morphism and $(\prod_{i\in I}G_i,\pi_i)$ satisfies the universal property of products for the category of hypergroups. Moreover, adapting the usual construction of filtered limits in the category of groups (see for instance, \cite{ribes2000profinite} or \cite{borceux1994handbook1}) we get the filtered limits for hypergroups under a technical hypothesis.


\begin{definition}
 A hypergroup $G$ has the \textbf{weak cancellation property} if for all $a,b,c\in G$, $1\in(ba^{-1})(ac)$ imply $b=c^{-1}$.
\end{definition}

\begin{proposition}
 Let $I$ be a directed poset and $\{G_i,\varphi_{ij}, I\}$ be a projective system of hypergroups. If, for all $i\in I$, $G_i$ has the weak cancellation property then there exist the filtered limit of $\{G_i,\varphi_{ij}, I\}$.
\end{proposition}
\begin{proof}
Let $I$ be a directed poset and $\{G_i,\varphi_{ij}, I\}$ be a projective system of hypergroups. In other words, we have, for each $i\ge j$ in $I$, a morphism $\varphi_{ij}:G_i\rightarrow G_j$ such that for all $i\ge j\ge k$ the following diagram commute:
$$\xymatrix{G_i\ar[rr]^{\varphi_{ik}}\ar[dr]_{\varphi_{ij}
} & & G_k \\ & G_j\ar[ur]_{\varphi_{jk}} &}$$
Now, let
$$G:=\left\lbrace(g_i)_{i\in I}\in\prod_{i\in I}G_i:\varphi_{ij}(g_i)=g_j\mbox{ for all }i\ge j\right\rbrace$$
and for $(g_i)_{i\in I},(h_i)_{i\in I}\in G$, define
$$(g_i)_{i\in I}\ast(h_i)_{i\in I}:=\left[(g_i)_{i\in I}\ast_{\prod_{i\in I}G_i}(h_i)_{i\in I}\right]\cap G.$$
We have that $(G,\ast)$ is a subgroup of $\prod_{i\in I}G_i$ (not full in general). 
For this, first note that $(1_i)_{i\in I}\in G$ and if $(g_i)_{i\in I}\in G$, since for all $i\ge j$ we have $\varphi_{ij}(g_i)=g_j$, then for all $i\ge j$ we have $\varphi_{ij}(g_i^{-1})=g_j^{-1}$, and $(g_i^{-1})_{i\in I}\in G$. Moreover, by the very definition of $\ast$ on $G$ we get that 
$$(h_i)_{i\in I}\in(1_i)_{i\in I}\ast(g_i)_{i\in I}\mbox{ iff }g_i=1_i\mbox{ for all }i\in I.$$
Now let $(p_i)_{i\in I}\in (g_i)_{i\in I}\ast(h_i)_{i\in I}$. Then $p_i\in g_i\ast h_i$ for all $i\in I$ and for $i\ge j$, $\varphi_{ij}(p_i)=p_j$. But then, $g_i\in p_i\ast h_i^{-1}$ for all $i\in I$ and $\varphi_{ij}(h_i^{-1})=h_j^{-1}$ for all $i\ge j$ (since $(h_i)_{i\in I}\in G$). Then $(g_i)_{i\in I}\in (p_i)_{i\in I}\ast(h_i^{-1})_{i\in I}$. We use a similar argument to prove that $(h_i)_{i\in I}\in (g_i^{-1})_{i\in I}\ast(p_i)_{i\in I}$.
Finally, let
$$(x_i)_{i\in I}\in [(a_i)_{i\in I}\ast(b_i)_{i\in I}]\ast(c_i)_{i\in I}.$$
Then $(x_i)_{i\in I}\in (e_i)_{i\in I}\ast(c_i)_{i\in I}$ for some $(e_i)_{i\in I}\in (a_i)_{i\in I}\ast(b_i)_{i\in I}$. Hence $x_i\in e_i\ast c_i$ with $e_i\in a_i\ast b_i$ for all $i\in I$, which means 
$$x_i\in (a_i\ast b_i)\ast c_i=a_i\ast(b_i\ast c_i).$$
Then for all $i\in I$, $x_i\in a_i\ast g_i$ for some $g_i\in b_i\ast c_i$. Now, for $i\ge j$ we have $\varphi_{ij}(x_i)=x_j$ and
$$x_j\in \varphi_{ij}(a_i\ast g_i)\subseteq
\varphi_{ij}(a_i)\ast \varphi_{ij}(g_i)=a_j\ast \varphi_{ij}(g_j).$$
We also have $x_j\in a_j\ast g_j$, which imply $x_j^{-1}\in g_j^{-1}\ast a_j^{-1}$. Then
\begin{align*}
    1&\in x_j^{-1}x_j\subseteq (g_j^{-1}\ast a_j^{-1})(a_j\ast\varphi_{ij}(g_i)).
\end{align*}
Since $G_i$ has the weak cancellation property we have $g_j^{-1}=\varphi_{ij}((g_i))^{-1}$, which means $g_j=\varphi_{ij}(g_i))$. Then $(g_i)_{i\in I}\in G$ and we have
$$[(a_i)_{i\in I}\ast(b_i)_{i\in I}]\ast(c_i)_{i\in I}\subseteq
(a_i)_{i\in I}\ast[(b_i)_{i\in I}\ast(c_i)_{i\in I}].$$
With a similar argument we conclude the reverse inclusion and then $G$ is a subgroup of the product. Denote by $\pi_i:\prod_{i\in I}G_i\rightarrow G_i$ the natural projection and $\psi_i:=\pi_i|_G$. We have that $\{G,\psi_i\}$ is compatible with $\{G_i,\varphi_{ij}, I\}$ in the sense that, for all $i\ge j$, $\varphi_{ij}\circ\psi_i=\psi_j$. Now, suppose that $\{H,\rho_i\}$ is compatible with $\{G_i,\varphi_{ij}, I\}$ and define 
$$\rho:H\rightarrow\prod_{i\in I}G_i$$
by the rule $\rho(x)=(\rho_i(x))_{i\in I}$. We have that $\rho$ is a morphism and since $\{H,\rho_i\}$ is compatible, for all $x\in H$ and all $i\ge j$ in $I$ we have $\varphi_{ij}(\rho_i(x))=\rho_j(x)$. This means that $(\rho_i(x))_{i\in I}\in G$ and $\mbox{Im}(\rho)\subseteq G$. Then we can consider $\rho:H\rightarrow G$, which proves that in fact,
$$G=\varprojlim_{i\in I}G_i.$$
\end{proof}


If $f:G\rightarrow H$ is a morphism, we already know that $\mbox{Ker}(f)$ is a full subgroup of $G$. It is straightforward to verify that $(\mbox{Ker}(f),\iota)$ satisfies the universal property of kernel for the category of hypergroups (here $\iota:\mbox{Ker}(f)\rightarrow G$ is the inclusion morphism, which is full). Then we have the following:

\begin{proposition}\label{prop:hypergroups_have_kernels}
 Every morphism in the category of hypergroups has a kernel.
\end{proposition}
\begin{proof}
Let $f:G\rightarrow H$ be a morphism. We already know that $\mbox{Ker}(f)$ is a full subgroup of $G$. Suppose that $E$ is another hypergroup and $t:E\rightarrow G$ is a morphism such that $0\circ t=f\circ t$.
$$\xymatrix@!=3pc{E\ar[r]^t & G\ar@<.7ex>[r]^-{f}\ar@<-.7ex>[r]_-{0} & H}$$
We have to show that there is a unique morphism $\overline t:E\rightarrow\mbox{Ker}(f)$ commuting the following diagram:
$$\xymatrix@!=4pc{
E\ar[dr]^t\ar@{.>}[d]_{\overline t} & & \\
\mbox{Ker}(f)\ar[r]^\iota & G\ar@<.7ex>[r]^-{f}\ar@<-.7ex>[r]_-{0} & H}$$
where $\iota:\mbox{Ker}(f)\rightarrow G$ is the (full) inclusion morphism. Since $0\circ t=f\circ t$, for all $x\in G$ we have $f(t(x))=0$ and $t(x)\in\mbox{Ker}(f)$. Hence we can consider $t:E\rightarrow\mbox{Ker}(f)$ and take $\overline{t}=t$. If $t':E\rightarrow\mbox{Ker}(f)$ is such that $t=\iota\circ t'$, then for all $x\in G$ we have $t(x)=\iota(t'(x))=t'(x)$, and thus $t'=t$.
\end{proof}

\begin{remark}
Every kernel is a monomoprhism since it is an equalizer. Remind that in the category of groups, we always can obtain the equalizer $\mbox{Eq}(f,g)$ from the kernel $\mbox{Ker}(f-g)$, but for hypergroups this property does not hold, because it is not true that $f\circ h = g \circ h $ if and only if $(f-g)\circ h = 0$, i.e, there is not ``only one morphism $f-g$".
\end{remark}

\begin{proposition}\label{injfull}
Let $f:G\rightarrow H$ be a full morphism of hypergroups. Then $f$ is injective if and only if $\mbox{Ker}(f)=\{1\}$.
\end{proposition}
\begin{proof}
($\Rightarrow$) Follow immediately. For ($\Leftarrow$), let $a,b\in G$ such that $f(a)=f(b)$. Then $1\in f(a)\ast f(b)^{-1}$ and since $f$ is full we have
$1\in f(a)\ast f(b)^{-1}=f(a\ast b^{-1})$. So there exist $c\in a\ast b^{-1}$ such that $f(c)=1$. Then $c\in \mbox{Ker}(f)=\{1\}$. So $a=(b^{-1})^{-1}=b$.
\end{proof}


Next, we restrict our attention to the category of abelian hypergroups in order to take quotients and analyze cokernels and coproducts.

\begin{proposition}\label{prop:hyper_group_have_cokernel}
 Every full morphism in the category of abelian hypergroups has a cokernel.
\end{proposition}
\begin{proof}
If $f:G\rightarrow H$ is a full morphism then $\mbox{Im}(f)$ is a full subgroup (by the very definition of full morphisms). Let $\pi:H\rightarrow H/\mbox{Im}(f)$ be the canonical projection. Suppose that $E$ is another abelian hypergroup and $t:H\rightarrow E$ is a morphism such that $t\circ 0=t\circ f$.
$$\xymatrix@!=3pc{G\ar@<.7ex>[r]^-{f}\ar@<-.7ex>[r]_-{0} & H\ar[r]^t & E }$$
We have to show that there exists a unique morphism $\overline t:H/\mbox{Im}(f)\rightarrow E$ commuting the following diagram:
$$\xymatrix@!=4pc{G\ar@<.7ex>[r]^-{f}\ar@<-.7ex>[r]_-{0} & H\ar[dr]^t\ar[r]^{\pi} & H/\mbox{Im}(f)\ar@{.>}[d]^{\overline t} \\
& & E}$$
Define\footnote{Remember that $[x]:=x+\mbox{Im}(f)$.} $\overline t([x]):=t(x)$. If $[x]=[y]$ then $x+\mbox{Im}(f)=y+\mbox{Im}(f)$ and $x+r=y+s$ for some $r,s\in\mbox{Im}(f)$, saying, $r=f(r_0)$ and $s=f(s_0)$. Since $t\circ 0=t\circ f$ we have 
$$t(x+r)\subseteq t(x)+t(r)=t(x)+t(f(r_0))=t(x)+0=t(x).$$
Then $t(x+r)=t(x)$. Similarly we have $t(y+s)=t(y)$. Then
$$t(x)=t(x+r)=t(y+s)=t(y)$$
and we prove that $[x]=[y]$ imply $t(x)=t(y)$. In particular the rule $\overline{t}$ does not depend on the choice of representatives. Now let $[d]\in [x]+[y]$. Then there exists $d'\in x'+y'$ with $[d']=[d]$, $[x']=[x]$ and $[y']=[y]$. Hence $t(d')\in t(x')+t(y')$ which implies $t(d)\in t(x)+t(y)$ and then, $\overline t[d]\in\overline t[x]+\overline t[y]$. Therefore, $\overline t$ is a morphism. By the construction we have $t=\overline t\circ\pi$. If $g:H/\mbox{Im}(f)\rightarrow E$ is another full morphism such that $t=g\circ\pi$, we have $g=\overline t$ (basically because $\pi$ is surjective).
\end{proof}



Moreover, adapting the usual construction of directed colimits in the category of abelian groups (see for instance, \cite{ribes2000profinite} or \cite{borceux1994handbook1}) and the argument used in \cite{ribeiro2021anel} we get the following:

\begin{proposition}
 The category of abelian hypergroups has directed colimits.
\end{proposition}
\begin{proof}
Let $I$ be a directed poset and $\{G_i,\varphi_{ij}, I\}$ be an injective system of hypergroups. In other words, we have for all $i\le j$ in $I$ a morphism $\varphi_{ij}:G_i\rightarrow G_j$ such that for all $i\le j\le k$ the following diagram commute:
$$\xymatrix{G_i\ar[rr]^{\varphi_{ik}}\ar[dr]_{\varphi_{ij}
} & & G_k \\ & G_j\ar[ur]_{\varphi_{jk}} &}$$
Now, let $U=\dot\bigcup_{i\in I}G_i$ (the disjoint union of $G_i$'s). For $x_i\in G_i$ and $x_j\in G_j$, we say that $x_i\sim x_j$ if there exists $k\ge i,j$ in $I$ such that $\varphi_{ik}(x_i)=\varphi_{jk}(x_j)$. This is an equivalence relation. For $\overline x,\overline y\in U/\sim$, define
$$\overline x+\overline y:=\{\overline z:\mbox{there exist }z'\sim z,\,x'\sim x,\,y'\sim y\mbox{ with }z'\in\varphi_{ik}(x')+\varphi_{jk}(y'),\,k\ge i,j\}.$$
We prove that $(U/\sim,+,\overline0)$ is an abelian hypergroup. Of course, $+$ is commutative and if $\overline y\in\overline0+\overline x$, then
$y'\in\varphi_{ik}(0)+\varphi_{jk}(x')=0+\varphi_{jk}(x')=\varphi_{jk}(x')$ for some $y'\sim y$ and $x'\sim x$. Then $y'\sim\varphi_{jk}(x')\sim x'\sim x$ (because $\varphi_{kk}=id_{A_k}$) and $y\sim x$, which means $\overline y=\overline x$.

Now let $\overline d\in\overline x+\overline y$. Then
$d'\in\varphi_{ik}(x')+\varphi_{jk}(y')$ for some $x'\sim x$, $y'\sim y$ and $d'\sim d$. Since $A_k$ is an abelian hypergroup, we have
$$\varphi_{ik}(x')\in d'-\varphi_{jk}(y')=\varphi_{kk}(d')-\varphi_{jk}(y'),$$
which means $\overline{\varphi_{ik}(x')}\in\overline{d'}-\overline{\varphi_{jk}(y')}$, which turns out to be $\overline x\in\overline d-\overline y$.

Finally, let $\overline d\in(\overline x+\overline y)+\overline z$. Then $\overline d\in\overline e+\overline z$ for some $\overline e\in\overline x+\overline y$. Then, we have $d'\sim d$, $e'\sim e$, $z'\sim z$ with $d'\in\varphi_{ik}(e')+\varphi_{jk}(z')$ and $e''\sim e$, $x'\sim x$, $y'\sim y$ with $e''\in\varphi_{pr}(x')+\varphi_{qr}(y')$. Now, since $e''\sim e'\sim\varphi_{ik}(e')$, we can choose $t\ge i,j,k,p,q,r$ such that $\varphi_{kt}(\varphi_{ik}(e'))=\varphi_{rt}(e'')$. Then
$$\varphi_{kt}(d')\in\varphi_{kt}(\varphi_{ik}(e'))+\varphi_{kt}(\varphi_{jk}(z'))\mbox{ and }\varphi_{rt}(e'')\in\varphi_{rt}(\varphi_{pr}(x'))+\varphi_{rt}(\varphi_{qr}(y')).$$
Then
$$\varphi_{kt}(d')\in[\varphi_{rt}(\varphi_{pr}(x'))+\varphi_{rt}(\varphi_{qr}(y'))]+\varphi_{kt}(\varphi_{jk}(z')).$$
Since $A_t$ is a hypergroup, we have
$$[\varphi_{rt}(\varphi_{pr}(x'))+\varphi_{rt}(\varphi_{qr}(y'))]+\varphi_{kt}(\varphi_{jk}(z'))
=\varphi_{rt}(\varphi_{pr}(x'))+[\varphi_{rt}(\varphi_{qr}(y'))+\varphi_{kt}(\varphi_{jk}(z'))].$$
Then $\varphi_{kt}(d')\in\varphi_{rt}(\varphi_{pr}(x'))+f$ for some 
$f\in \varphi_{rt}(\varphi_{qr}(y'))+\varphi_{kt}(\varphi_{jk}(z'))$. This means $\overline{\varphi_{kt}(d')}\in\overline{\varphi_{rt}(\varphi_{pr}(x'))}+\overline f$ with $\overline f\in \overline{\varphi_{rt}(\varphi_{qr}(y'))}+\overline{\varphi_{kt}(\varphi_{jk}(z'))}$ which turns out to be $\overline d\in\overline x+\overline f$ with $\overline f\in\overline y+\overline z$, and we prove that 
$$(\overline x+\overline y)+\overline z\subseteq
\overline x+(\overline y+\overline z).$$
The other inclusion follows by a similar argument.

Thus $U/\sim$ is an abelian hypergroup which we will denote by $G:=U/\sim$. For each $g_i\in G_i$, define $\psi_i:G_i\rightarrow G$ by $\psi_i(g_i)=\overline g_i$. This is a morphism and $\{G_i,\psi_i\}$ is compatible with $\{G_i,\varphi_{ij}, I\}$ in the sense that for all $i\le j$, $\psi_j\varphi_{ij}=\psi_i$.

Now suppose that $\{H,\rho_i\}$ is compatible with $\{G_i,\varphi_{ij}, I\}$. We define a morphism $\rho:G\rightarrow H$ by the following rule: for $a\in U$ with $a=\overline x$ for some $x\in G_i$, define $\rho(\overline x)=\rho_i(x)$. Then $\rho$ is the unique morphism such that $\rho\psi_i=\rho_i$ for all $i\in I$. So we have
$$G=\varinjlim_{i\in I}G_i.$$
\end{proof}

Let $\{G_i\}_{i\in I}$ be a family of abelian hypergroups. Define
$$\bigoplus_{i\in I}G_i:=\left\lbrace(a_i)_{i\in I}\in\prod_{i\in I}G_i:\mbox{ thre exist }n_0\mbox{ such that }c_i=0\mbox{ for all }i\ge n_0\right\rbrace.$$

In fact, $\bigoplus_{i\in I}G_i$ is a full subgroup of $(\prod_{i\in I}G_i,\ast,(1)_{i\in I})$. Note that $\bigoplus_{i\in I}G_i$ is not the coproduct in general, but we use it to define a weak biproduct:

\begin{definition}
 Let $\mathcal{C}$ be a category with zero morphisms. The \textbf{weak biproduct} of a finite collection of objects $A_1, ..., A_n$ in $\mathcal{C}$ is an object  $A_1\oplus \dots \oplus A_n$ together with:
 
\begin{itemize}
    \item projection morphisms $p_k:A_1\oplus \dots \oplus A_n \to A_k $
    \item injection morphisms $\iota_k : A_k \to A_1\oplus \dots \oplus A_n$
\end{itemize}
such that
\begin{itemize}
    \item $p_k \circ \iota_k = id_{A_k}$
    \item $p_l \circ \iota_k = 0,$ the zero morphism $A_k \to A_l$, for $k \neq l$\footnote{In a category with a zero object $\textbf{1}$, the zero morphism $0_{AB}:A \to B$ between two objects $A, B$ is the unique morphism that factors through $\textbf{1}$.};
    \item $(A_1\oplus \dots \oplus A_n,p_1,...,p_n)$ is a product;
    \item $id_{A_1\oplus \dots \oplus A_n}\subseteq i_1\circ p_1+\dots +i_n\circ p_n$.
\end{itemize}
\end{definition}



\begin{lemma}
 The category of abelian hypergroups has weak biproducts.
\end{lemma}

\section{Hyperbolic Hypergroups}\label{sec:hyperbolic}

In the context of abstract theories of quadratic forms, there is a certain kind of multi-algebra that arises naturally, the hyperbolic hyperrings/hyperfields (see, for instance, \cite{roberto2021quadratic}  and \cite{roberto2021ktheory}). Here, we generalize this for hypergroups.

\begin{definition}
 A hypergroup $G$ is \textbf{hyperbolic} if for all $a\in G\setminus\{1\}$, $a\ast a^{-1}=G$. The category of hyperbolic hypergroups and its morphisms will be denoted by $\mbox{HHG}$.
\end{definition}

Every hyperbolic hypergroup $G$ is \textbf{rooted} in the sense that if $a,b\in G\setminus\{1\}$ then $a,b\in a\ast b$. In fact, a hypergroup $G$ if rooted if and only if is hyperbolic. In fact, if $a,b\in a\ast b$ for all $a,b\in G\setminus\{1\}$ then $b\in a\ast a^{-1}$. On the other hand, if $a\ast a^{-1}=G$ for all $a\in G\setminus\{1\}$, then for all $a,b\in G\setminus\{1\}$  we get $b\in a^{-1}\ast a$ and $a\in b^{-1}\ast b$. This implies $a\in a\ast b$ and $b\in a\ast b$. Then we use ``hyperbolic" and ``rooted" as synonyms.

Let $G_1$ and $G_2$ be two hyperbolic hypergroups. We define a new hyperbolic hypergroup $(G_1\times_h G_2,\ast,r,(1,1))$ by the following: the adjacent set of this structure is 
$$G_1\times_h G_2:=(G_1\setminus\{1\}\times G_2\setminus\{1\})\cup\{(1,1)\}.$$
For $(a,b),(c,d)\in G_1\times_h G_2$ we define
\begin{align}\label{prodmultiop}
    r(a,b)&=(r(a),r(b)),\nonumber\\
    (a,b)\ast(c,d)&=\{(e,f)\in G_1\times G_2:e\in a\ast c\mbox{ and }f\in b\ast d\}\cap(G_1\times_h G_2).
\end{align}
In other words, $(a,b)\ast(c,d)$ is defined in order to avoid elements of $G_1\times G_2$ of type $(x,1),(1,y)$, $x,y\ne1$. Using the very same argument of Theorem 2.17 in \cite{roberto2021ktheory} we obtain the following.

\begin{theorem}[Product of Hyperbolic Hypergroups]\label{hfproduct}
Let $G_1,G_2$ be hyperbolic hypergroups and $G_1\times_h G_2$ as above. Then $G_1\times_h G_2$ is a hyperbolic hypergroup and satisfy the Universal Property of product for $G_1$ and $G_2$.
\end{theorem}

In order to avoid confusion and mistakes, we denote the binary product in $\mbox{HHG}$ by $G_1\times_hG_2$. For hyperbolic hypergroups $\{G_i\}_{i\in I}$, we denote the product of this family by
$$\prod^h_{i\in I}G_i,$$
with underlying set defined by
$$\prod^h_{i\in I}G_i:=\left(\prod_{i\in I}\dot G_i\right)\cup\{(1_i)_{i\in I}\}$$
and operations defined by rules similar to the ones defined in \ref{prodmultiop}. If $I=\{1,...n\}$, we denote
$$\prod^h_{i\in I}G_i=\prod^n_{\substack{i=1 \\ [h]}}G_i.$$

Of course, $G_1\times_hG_2\ne G_1\times G_2$ so in the category of hyperbolic hypergroups we do not necessarily have coproducts. We observe that in $\mbox{HHG}$ we have nice properties for morphisms, such as the following.

\begin{theorem}\label{caracter}
Let $G$ be a hyperbolic hypergroup. There is a bijection between $G$ and morphisms $\varphi:Q_2\rightarrow G$.
\end{theorem}
\begin{proof}
For  $a\in G$, we have a morphism $\varphi_a:Q_2\rightarrow G$ given by the rule
$$\varphi_a(k)=\begin{cases}a\mbox{ if }k=1\\1\mbox{ if }k=0\\-a\mbox{ if }k=-1\end{cases}.$$
Now, define $\Phi:G\rightarrow\mbox{HHG}(Q_2,G)$ by the rule $\Phi(a)=\varphi_a$. Then $\Phi$ is the desired bijection.
\end{proof}

\begin{theorem}\label{hmono}
Let $f:G\rightarrow H$ be a morphism in the category of hyperbolic hypergroups. Then $f$ is a mono if and only if $f$ is injective.
\end{theorem}
\begin{proof}
($\Rightarrow$) Let $\varphi_a$ as in Theorem \ref{caracter}.
Of course, the fact of $\varphi_a$ be a morphism for $a\ne1$ is consequence of $G$ be a hyperbolic hypergroup. Moreover $\varphi_a=\varphi_b$ if and only if $a=b$. Now let $a,b\in G$ with $a\ne b$. If $f(a)=f(b)$ then $f\circ\varphi_a=f\circ\varphi_b$. Since $f$ is mono we have $\varphi_a=\varphi_b$, which means $a=b$. Hence $f$ is injective. \newline
($\Leftarrow$) Let $f$ injective and $g_1,g_2:A\rightarrow G$ such that $f\circ g_1=f\circ g_2$. Then for all $x\in A$, $f(g_1(x))=f(g_2(x))$. Since $f$ is injective this means $g_1(x)=g_2(x)$ for all $x\in A$ and then $g_1=g_2$. 
\end{proof}

Another interesting property is regarding the kernels.

\begin{proposition}
 Let $f:G\rightarrow H$ be a morphism in $\mbox{HHG}$. Then $$\mbox{Ker}(f)=\{0\}\mbox{ or }\mbox{Ker}(f)=G.$$
\end{proposition}
\begin{proof}
If fact, if $a\ne0$ is in $\mbox{Ker}(f)$ then $-a\in\mbox{Ker}(f)$. Since $\mbox{Ker}(f)$ is a full subgroup of $G$, we have $a-a\subseteq \mbox{Ker}(f)$ and since $G$ is hyperbolic, $G=a-a\subseteq\mbox{Ker}(f)$.
\end{proof}

Using Proposition \ref{injfull} we get the following useful characterization:

\begin{corollary}
Let $f:G\rightarrow H$ be a full morphism in $\mbox{HHG}$. Then $f=0$ or $f$ is injective.
\end{corollary}

\section{Hyper-abelian Categories}\label{sec:hyper-abelian}
In this Section we seek for a generalization of the notion of an abelian category to encompass hyper algebraic structures. Let $G,H$ be hypergroups and denote
$$\mbox{Hom}(G,H):=\{f:G\rightarrow H:f\mbox{ is a morphism}\}.$$
For $f,g\in \mbox{Hom}(G,H)$, define
$$f\ast g:=\{h\in\mbox{Hom}(G,H):h(x)\in f(x)\ast g(x)\mbox{ for all }x\in G\}.$$

For all $x \in G$, and any morphism $f \in \mbox{Hom}(G,H)$ $$r(f)(x) = f(x)^{-1}$$
and 
$$1(x) = 1$$
That is, the identity in $\mbox{Hom}(G,H)$ is the constant morphism that always gives the identity of $H$.

\begin{remark}
A non-associative group is the same as a \textit{quasigroup} with (two-sided) identity element (or a invertible loop). When the multivalued operation $*$ of the hypergroup does not satisfies the associativity rule, we will call it \textbf{nonassociative hypergroup}.
\end{remark}

\begin{lemma}
With the above definition, $(\mbox{Hom}(G,H),\ast,r,1)$ is a nonassociative hypergroup.
\end{lemma}
\begin{proof}
    Let $h, f, g \in \mbox{Hom}(G,H)$ and $x \in A$ 
\begin{enumerate}[i -]
    \item  If $h \in f \ast g $, then $h(x) \in f(x) \ast g(x)$. Since $H$ is hypergroup:
    \begin{align*}
        f(x) & \in h(x) \ast (g(x))^{-1} =  h(x) \ast r(g)(x) \mbox{ implies } f \in h \ast r(g) \\
        g(x) & \in (f(x))^{-1} \ast h(x) = r(f)(x) \ast h(x) \mbox{ implies } g\in r(f) \ast h
    \end{align*}
    
    \item Since $H$ is hypergroup
    \begin{align*}
        g \in 1 \ast f & \mbox{ iff } g(x) \in 1(x) \ast f(x) = 1 \ast f(x)    \\
        & \mbox{ iff } g(x) = f(x)  \\
        & \mbox{ iff } f = g
    \end{align*}    
\end{enumerate}
\end{proof}
    
    Now we explain why the associativity does not hold in general: let $d\in (f\ast g)\ast h$. Then $d\in e\ast h$ for some $e\in f\ast g$. This means that for all $x\in G$, $d(x)\in e(x)\ast h(x)$ and $e(x)\in f(x)\ast g(x)$. Then 
    $$d(x)\in (f(x)\ast g(x))\ast h(x)=f(x)\ast (g(x)\ast h(x)).$$
    Hence, there is a function $w:G\rightarrow H$ such that 
    $$d(x)\in f(x)\ast w(x)\mbox{ and }w(x)\in g(x)\ast h(x)\mbox{ for all }x\in G.$$
    However, there is no reason to $w$ be a morphism.

\begin{definition}
 A category $C$ is \textbf{hyper-almost- preadditive} if each $Hom(G,H)$ in $C$ is a non-associative hypergroup and the composition of the morphisms is bilinear, i.e., the map
 	\begin{align*}
	    Hom(F,G) \times Hom(G,H)  &\longrightarrow  Hom(F,H) \\
	    (f,g)  &\mapsto  g \circ f
	\end{align*}
is such that 
   		$$(g\ast g')\circ (f \ast f') \subseteq g \circ f \ast g \circ f' \ast g' \circ f \ast g' \circ f' $$
\end{definition}

If $Hom(G,H)$ in the above definition is a hypergroup, then we say $C$ is a hyper-preadditive category. From now on, a hyper-almost-preadditive category will be called just almost-preadditive category if the context allows it. Next, we show what remains to obtain that the category of hypergroups is an example of a almost-preadditive category.

\begin{lemma}\label{lem:hypergroups_are_preadditive}
Let $G,H$ be hypergroups. For $f,f' \in Hom(F,G)$ and $g, g' \in Hom(G,H)$ we have the bilinearity $$(g\ast g')\circ (f \ast f') \subseteq g \circ f \ast g \circ f' \ast g' \circ f \ast g' \circ f'  $$
\end{lemma}
\begin{proof}
Let $a\in g\ast g'$ and $b\in f\ast f'$. Now, for all $x\in F$ we have
$b(x)\in f(x)\ast f'(x)$ and $a(b(x))\in g(b(x))\ast g'(b(x))$. Then
\begin{align*}
    a(b(x))&\in g(b(x))\ast g'(b(x))\subseteq 
    [g(f(x)\ast f'(x))]\ast [g'(f(x)\ast f'(x))] \\
    &\subseteq[g(f(x))\ast g(f'(x))]\ast [g'(f(x))\ast g'(f'(x))]\\
    &=(g\circ f)(x)\ast(g\circ f')(x)\ast(g'\circ f)(x)\ast(g'\circ f')(x).
\end{align*}
\end{proof}
If all these morphisms are full, then we have the equality
$$[g\circ (f\ast f')]\ast [g'\circ (f\ast f')]=
    g \circ f \ast g \circ f' \ast g' \circ f \ast g' \circ f',$$
but we only have $(g\ast g')\circ (f \ast f') \subseteq[g\circ (f\ast f')]\ast [g'\circ (f\ast f')]$ and the equality does not hold in general. 

Therefore: 
\begin{corollary}
The category of hypergroups is hyper-almost-preadditive.
\end{corollary}

$ $

\begin{remark}
Aiming at a closer relation with abelian homological algebra we would like to have that the category of hypergroups is hyper-preadditve, that is, the set $Hom(G,H)$ is a hypergroup. In the future, we may explore the consequences of the nonassociativity in $Hom(G,H)$ (it is possible that we should be guided by the theory of semi-abelian categories instead of abelian categories) and other notions of morphism between hypergroups with a different multivalued operation, hoping to $Hom(G,H)$ be a hypergroup.
\end{remark}

We recall that a category is \textbf{additive} if it is preadditive,  has a zero objects, and has binary products\cite{borceux1994handbook2}. So we define a \textbf{hyper-almost-additive} category as a hyper-almost-preadditive that has a zero object, and binary weak biproducts.
\begin{remark}
The category of abelian hypergroups is a almost-additive-category.
\end{remark}

We remind the categorical definition of the image of a morphism: 

\begin{definition}[Categorical Image]
The \textbf{categorical image} of a morphism $f: A \to B$ in category $\mathcal C$ is a monomorphism $m: I \to B$  such that: 
\begin{enumerate}
    \item There exists a morphism $e: A \to I$ such that $f = m \circ e;$
    \item If $I'$ is an object in $\mathcal{C}$ with a morphism $e': A \to I'$ and a mono $m': I' \to B$ such that $f = m' \circ e'$, there is a unique morphism $\phi: I \to I'$ such that $m = m' \circ \phi.$
\end{enumerate}
\end{definition}

$ $

\begin{proposition}\label{prop:im_is_cat_image}
 In the category of hyperbolic hypergroups, if $f$ is a full morphism, then $\mbox{Im}(f)$ is the (categorical) image.
\end{proposition}
\begin{proof}
Let $f: A \to B$ be a full morphism in the category of hypergroups. We defined the image of $f$ by $\mbox{Im}(f) := \{b \in B \,:\, b = f(a) \mbox{ for some } a \in A\} $ where we have that the inclusion morphism $i: \mbox{Im}(f)  \to B$ is full, since $\mbox{Im}(f)$ is a full subgroup (because $f$ is full). \newline
First we show that $i$ is a monomorphism: take $g,h: X \to \mbox{Im}(f)$ such that $i\circ g = i \circ h$. So $i(g(x)) = i(h(x))$, for all $x \in X$. Since $i$ is the inclusion, $g(x) = h(x)$, for all $x \in X$. So $g=h.$ \newline
Now, we show that $\mbox{Im}(f)$ with the inclusion satisfies the universal property of the categorical notion of image: take a hypergroup $I$ with a monomorphism $m: I \to B$ and a morphism $e: A \to I$ such that $f = m \circ e$. \newline
For every $a \in A,$ $f(a) = m(e(a)).$ Define $\phi: \mbox{Im}(f) \to I$ by $\phi(f(a)) = e(a)$ (which is well defined since $m$ is injective). If $y\in f(b)\ast f(c)=f(b\ast c)$, then $y=f(a)$ for some $a\in b\ast c$. Then $e(a)\in e(b\ast c)\subseteq e(b)\ast e(c)$, proving that $\phi$ is in fact a morphism. Moreover
$$m(\phi(f(a))) = m(e(a)) = f(a) = i(f(a))$$ as desired. \newline
Finally, suppose there is another $\psi: \mbox{Im}(f) \to I$ such that $i = m \circ \psi$. So $m \circ \psi = m \circ \phi$ implies $\psi = \phi$ because $m$ is mono.
\end{proof}




\begin{remark}
In the category of hypergroups with full morphisms, every morphism admits categorical image. However, if we consider the category of hypergroups with all morphisms (full and not full), we do not know if the morphisms that admit image are precisely the full morphisms.
\end{remark}


\begin{proposition}\label{prop:im=ker(coker)}
 In the category of hyperbolic hypergroups, if $f$ is a full morphism, then
 $$\mbox{Im}(f)\cong\mbox{Ker}(\mbox{Coker}(f)).$$
\end{proposition}
\begin{proof}
 Given $f: A \to B$ a full morphism of hypergroups, by \ref{prop:hyper_group_have_cokernel}, the morphism $coker(f)$ is given by the projection $\pi: B \to B/\mbox{Im}(f)$. Thus, by \ref{prop:hypergroups_have_kernels}, $\mbox{Ker}(\mbox{Coker}(f)) = \mbox{Ker}(\pi) = \{b \in B\,:\, \pi(b) = 0 \}$.
 So,
 \begin{align*}
     b \in \mbox{Im}(f) & \mbox{ iff } b + Im(f) = 0 
     \mbox{ iff } \pi(b) = 0
      \mbox{ iff } b \in \mbox{Ker}(\pi).
 \end{align*}
 Therefore, $\mbox{Im}(f) = \mbox{Ker}(\mbox{Coker}(f))$.
\end{proof}





\begin{definition}\label{haab}
 A pair $(\mathcal C,\mathcal F)$ is a \textbf{hyper-almost-abelian category} if $\mathcal C$ is hyper-almost-additive category and $\mathcal F$ is a subcategory of $\mathcal C$ satisfying the following properties:
\begin{description}
    \item [HA1 -] Every morphism in $\mathcal C$ has kernel;
    \item [HA2 -] Every morphism in $\mathcal F$ has cokernel in $\mathcal C$;
    \item [HA3 -] Every morphism in $\mathcal F$ has categorical image in $\mathcal F$;
    \item [HA4 -] For every morphism $f$ in $\mathcal F$ we have these isomorphisms in $\mathcal C$:
    $$\mbox{Im}(f)\cong\mbox{Ker}(\mbox{Coker}(f))\cong \mbox{Coker}(\mbox{Ker}(f))$$
\end{description}
If all these conditions hold for $(\mathcal C,\mathcal F)$, we say in short that $\mathcal C$ is an \textbf{abelian witness} of $\mathcal F$.
\end{definition}

Note that in Definition \ref{haab} we are not requiring that $\mathcal F$ is hyper-almost-preadditive. Also, in order to show that the category of abelian hypergroups is an abelian witness of the category of hyperbolic hypergroups and full morphisms, we only need to deal with the fourth property in Definition \ref{haab} (of course, we are using Propositions \ref{prop:im=ker(coker)}, \ref{prop:hypergroups_have_kernels}, and \ref{prop:hyper_group_have_cokernel}, and Lemma \ref{lem:hypergroups_are_preadditive}).

\begin{theorem}\label{theo:isotheoholds}
In the category of abelian hyperbolic hypergroups every full morphism $f$ induces an isomorphism $\mbox{Ker}(\mbox{Coker}(f))\cong \mbox{Coker}(\mbox{Ker}(f))$. 
\end{theorem}
\begin{proof}
Let $f: A \to B$ be a full morphism of hypergroups. Observe that Propostions \ref{prop:hypergroups_have_kernels} and \ref{prop:hyper_group_have_cokernel} provide that $\mbox{Coker}(\mbox{Ker}(f)) = A/Ker(f)$. By Propostion \ref{prop:im=ker(coker)}, we have that for any $a \in A$,  $\mbox{Im}(f) = \mbox{Ker}(\mbox{Coker}(f)) $. Then, we consider $\overline{f} : A/Ker(f) \to \mbox{Im}(f)$ defined by $\overline{f}([a]) = f(a)$, where $[a] = a + Ker(f)$. We will show that $\overline{f}$ if a full isomorphism of hypergroups.

First, note it does not depend of the choice of representatives: for $a, b \in A$,   suppose $[a] = [b]$. So, there are $x, y \in Ker(f)$ such that $a + x = b + y$. Then,
$$a \in b + y - x  \mbox{ implies } f(a) \in f(b + y -x) = f(b) + f(y) + f(-x) = f(b)$$

It is a morphism of hypergroups: take $a, b, d \in A$ such that $[d] \in [a] + [b]$, then we have $d' \in a' + b'$, for $[a] = [a']$, $[b] = [b']$, and $[d] = [d']$.  Since $f$ is a morphism, $f(d') \in f(a') + f(b')$. So 
$$\overline{f}([d]) = \overline{f}([d']) \in \overline{f}([a']) + \overline{f}([b'])  = \overline{f}([a]) + \overline{f}([b]) $$

We also have $0 = f(0) = \overline{f}([0])$, and $\overline{f}([-x]) = f(-x) = -f(x) = - \overline{f}([x])$.

It is surjective: Let $y \in Im(f)$. Thus, $y = f(x)$ for some $x \in A$. By definition of $\overline{f}$, $y = \overline{f}([x])$.

It is injective: suppose $\overline{f}([a]) = \overline{f}([b])$. So $f(a) = f(b)$ and then $0 \in f(a) - f(b) = f(a-b)$, since $f$ is full. Thus, there is $x \in a - b$ such that $f(x) = 0$, and $x \in a - b$ implies that $a \in b+x$. In on hand we have:
$$a \in b + x \mbox{ implies } a + Ker(f) \subseteq (b + x) + Ker(f) = b + (x+Ker(f)) \subseteq b + Ker(f) $$
On the other hand, $ x \in a - b \mbox{ implies } -x \in b -a \mbox{ which implies } b \in a -x $ and 
$$b \in a -x \mbox{ implies } b + Ker(f) \subseteq (a-x) + Ker(f) = a + (-x + Ker(f)) \subseteq a + Ker(f) $$
where we used that $-x \in Ker(f).$ Indeed, if $x \in Ker(f)$ then $0 \in x -x $ implies that $f(0) \in f(x) + f(-x)$ and so $0 \in 0 + f(-x) = \{f(-x)\}.$
\end{proof}

In the same way that the category of abelian groups is the prototypical example of an abelian category, we established that the category of abelian hypergroups (with a convenient subcategory $\mathcal F$) is the prototypical example of an hyper-almost-abelian category. 

\begin{example}
Since every group is a hypergroup and in abelian categories every morphism has categorical image, by definition, we conclude that for every abelian category $\mathcal{A}$, the pair $(\mathcal{A},\mathcal{A})$ is an hyper-almost-abelian category.
\end{example}
 
 Using all previous results we conclude the following.
 
 \begin{example}
 Let $\mbox{HHG}_f$ be the category of hyperbolic hypergroups with full morphisms. Then $(\mbox{HG},\mbox{HHG}_f)$ is an almost (hyper) abelian category.
 \end{example}
 
 The key argument in the proof of Theorem \ref{prop:im_is_cat_image} is the fact that in $\mbox{HHG}_f$ every monomorphism is an injective morphism. Using this we have another example of almost (hyper) abelian category.
 
 \begin{example}[Idempotent Hypergroups]
 A hypergroup $G$ will be called \textbf{idempotent} if for all $a\in G$ $r(a)=a$ and $\{1,a\}\subseteq a\ast a$. The category of idempotent hypergroups and their morphisms will be denoted by $\mbox{IHG}$.
 
 If $G$ is idempotent, there is a bijection between $G$ and morphisms $\varphi:K\rightarrow G$. In fact, for $a\in G$, we have a morphism $\varphi_a:K\rightarrow G$ given by the rule
$$\varphi_a(k)=\begin{cases}a\mbox{ if }k=1\\1\mbox{ if }k=0\end{cases}.$$
Now, define $\Phi:G\rightarrow\mbox{HHG}(K,G)$ by the rule $\Phi(a)=\varphi_a$. Then $\Phi$ is the desired bijection. Using this bijection we conclude that every mono in $\mbox{IHG}$ is in an injective morphism (using a very similar argument to that used for $\mbox{HHG}$).
 
 So, we have that every full morphism $f:G\rightarrow H$ in $\mbox{IHG}$ has categorical image (just copy the proof of Proposition \ref{prop:im_is_cat_image}) and in this case (copying the proof of Theorem \ref{theo:isotheoholds}), we get these isomorphisms below
 $$\mbox{Im}(f)\cong\mbox{Ker}(\mbox{Coker}(f))\cong\mbox{Coker}(\mbox{Ker}(f)).$$
 Then considering $\mbox{IHG}_f$ as the category of idempotent hypergroups with full morphisms we have that $(\mbox{HG},\mbox{IHG}_f)$ is an almost (hyper) abelian category.
 \end{example}

Finally, we discuss about the existence of injective objects in our framework. Remind that in an arbitrary category $\mathcal{C}$ , an object $I$ of $\mathcal{C}$ is \textit{injective} if for every monomorphism $f: A \to B$ and every morphism $g: A \to I$ there is a morphism $h: B \to I$ such that $h\circ f = g$. If $\mathcal{C}$ is an abelian category, this is equivalent to say that the functor $Hom_\mathcal{C}(-,I)$ is exact, i.e., preserves short exact sequences. Moreover, if for every object $A$ of $\mathcal{C}$, there exists a monomorphism from $A$ to an injective object $I$, then we say that $\mathcal{C}$ has enough injectives. The category $\textbf{Ab}$ of abelian groups is an abelian category and it is useful to have that $I$ is injective if and only if $Hom_\mathcal{C}(-,I)$ is exact to show that $\textbf{Ab}$ has enough injectives. Furthermore, we use that the injectives in $\textbf{Ab}$ are precisely the divisible groups, in the presence of Zorn's Lemma. 

For hypergroups, we have to be careful because  a short exact sequence of hypergroups $0 \to X \xrightarrow{f} Y \xrightarrow{g} Z \to 0$ must require that $f$ is full so that $Im(f) = Ker(Coker f)$ exists. Besides, when we apply the functor $Hom(-,I)$ for some hypergroup $I$, we obtain a sequence $Hom(Z,I) \xrightarrow{-\circ g} Hom(Y,I) \xrightarrow{-\circ f} Hom(X,I)$ that is not anymore a sequence of hypergroups, and now we must have that $-\circ g$ is full to consider its image and be able to study the exactness of $Hom(-,I)$ for $I$ injective or not. Moreover, we need a general better understanding of exact functors, and exact sequences before considering the study of injective objects, and analyze when a hyper-almost-abelian category has enough injectives. Additionally, in the case of abelian hypergroups, we may investigate an adequate notion of divisible hypergroups.

\section{Remarks and Future Work}

The formulation of abelian categories by Grothendieck in \cite{grothendieck1957quelques} provided equal treatment for distinct (co)homology theories, such as sheaf and group cohomologies. Our aim is to provide a general framework for dealing with cohomology of sheaves over hypergroups/hyperrings. By doing so, we would contribute to advancements in Algebraic Geometry. For instance, the 'etale cohomology of a sheaf over a scheme is defined using the right derived functors of the section functor, which exist because the corresponding sheaf category is abelian with enough injectives. A scheme is a locally ringed space, consisting of a pair $(X, \mathcal{O}_X)$ where $X$ is a topological space and $\mathcal{O}_X$ is a sheaf of rings (the \textit{structure sheaf}), satisfying certain conditions. In the case of hypergroups, we believe that the corresponding sheaf category will not be abelian since the category of hyperrings is not preadditive, but rather hyper-almost-abelian. Thus, our approach would allow for the definition of \'etale cohomology of a sheaf over a \textit{hyper scheme}. Sheaves of hyperrings have been used in the context of Algebraic Geometry \cite{jun2015algebraic} and Abstract Quadratic Forms Theories \cite{ribeiro2021anel}.

Furthermore, the theory of (co)homology for hyper-almost-abelian categories that we are proposing will be of interest for other types of cohomology over hyperstructures, such as hypergroup cohomology or Hochschild cohomology of hyperalgebras. In both cases, having an appropriate notion of module over hypergroups and hyperrings, similar to the definition in \cite{ameri2003categories}, is an important step.

Our notion of almost-preadditive category heavily relies on the structure of the set of morphisms in $Hom(G,H)$, where $G$ and $H$ are hypergroups. It should be noted that there are different choices for morphisms between $G$ and $H$, such as the set of \textit{m.w-homomorphisms} proposed in \cite{ameri2003categories}. Different choices of morphisms and multivalued operations will result in different properties for the set $Hom(G,H),$ making a careful study of the various possibilities necessary. Additionally, the notion of an almost-abelian category that we introduced is a pair $(\mathcal{C}, \mathcal{F})$ where $\mathcal{C}$ serves as an abelian witness for the subcategory $\mathcal{F}$. In our case, $\mathcal{C}$ represents the category of abelian hypergroups, and $\mathcal{F}$ represents the category of full hyperbolic hypergroups. However, we believe that other subcategories $\mathcal{F}$ could also be ``abelian" under the witness of $\mathcal{C}$, particularly the category of full hypergroups ($FHGr$). Nevertheless, since we were unable to show that every monomorphism in $FHGr$ is an injective full morphism, we do not have Proposition \ref{prop:im_is_cat_image} for $FHGr$. Moreover, instead of full morphisms, one could consider the notions of strong and ideal morphisms, as available in \cite{ribeiro2016functorial}.

We also observe that the subcategory $\mathcal{F}$ corrects the behavior of cokernels in $\mathcal{C}$, while $\mathcal{C}$ corrects the behavior of coproducts in $\mathcal{F}$. To ensure the existence of cokernels, we utilize full morphisms. Our focus on hyperbolic hypergroups is due to the well-behaved nature of their corresponding category (for instance, we have proven that morphisms of hyperbolic hypergroups are monic exactly when they are injective). However, the category of full hyperbolic hypergroups by itself is not a suitable candidate for being a type of unary almost abelian category, as the binary coproducts do not coincide with the binary products. Given this perspective, we may consider a dual definition of hyper-almost-abelian categories and explore concrete examples.

Nevertheless, we believe that we are not far from developing a hyper (almost) abelian homological algebra. It should be noted that if $C$ and $C'$ are abelian categories with $C$ having enough injectives, then there exist right derived functors of a left-exact functor $F: C \to C'$ \cite[Section 2.5]{weibel1995introduction}. In other words, under these conditions, we have the basic ingredients for developing a cohomology theory. Therefore, a crucial next step is to study which categories of hyperstructures have enough injectives and determine if right derived functors still exist in the context of hyper-almost-abelian categories, rather than abelian categories.

Another course of action for future works is to understand if our notion of hyper-preadditive category can be see as an enriched category, adapting the definition of enriched category. Furthermore, since we have multiple possibilities of morphism between hypergroups, providing distinct properties to the set of morphisms $Hom(G,H)$, it is possible to study categories enriched not by one but by multiple categories.

Finally, we would like to reiterate that we have asserted that the category of hypergroups does not necessarily have coproducts, nor is it preadditive, because the binary operation in $(Hom(G,H), \ast, r, 1)$ is nonassociative. Recently, in \cite{nakamura2023categories}, the authors provided counterexamples that prove that their category of hypergroups also lacks coproducts and is not preadditive (refer to Theorems 4.20 and 4.23). It is important to note that our categories of hypergroups are distinct, as they have different morphisms. In the future, we plan to investigate whether their counterexamples can be applied to our case.


\section{Acknowledgments}

This study was financed by the Coordena\c{c}\~ao de
Aperfei\c{c}oamento de Pessoal de N\'ivel Superior - Brasil (CAPES) - Finance Code 001. Ana Luiza Ten\'orio is funded by CAPES. Grant Number 88882.377949/2019-01.
We also thanks to professor Hugo L. Mariano for his valuable suggestions, and the reviewers for the effort to review the manuscript.

\bibliographystyle{plain}
\bibliography{one_for_all}

\end{document}